\documentclass[11pt]{amsart}
\usepackage{geometry}
\usepackage[utf8]{inputenc}
\usepackage{amsmath,amsthm,amssymb,tikz-cd,enumerate,hyperref}

\numberwithin{equation}{section}
\newtheorem{thm}[equation]{Theorem}
\newtheorem{lemma}[equation]{Lemma}
\newtheorem{cor}[equation]{Corollary}
\newtheorem{prop}[equation]{Proposition}

\theoremstyle{definition}
\newtheorem{defn}[equation]{Definition}

\theoremstyle{remark}
\newtheorem*{rmk}{Remark}

\DeclareMathOperator{\TC}{\mathrm{TC}}
\newcommand{\from}{\colon}
\newcommand{\R}{\mathbb{R}}
\newcommand{\Z}{\mathbb{Z}}

\DeclareMathOperator{\genus}{genus}
\DeclareMathOperator{\cat}{cat}

\title{Relative Topological Complexity and Configuration Spaces}
\author{Bryan Boehnke \and Steven Scheirer \and Shuhang Xue}
\address{Department of Mathematics and Statistics, Carleton College, Northfield, MN 55057}
\email{\href{mailto:boehnkeb@carleton.edu}{boehnkeb@carleton.edu}, 
    \href{mailto:sscheirer@carleton.edu}{sscheirer@carleton.edu},
    \href{mailto:xues@carleton.edu}{xues@carleton.edu}}

\begin{document}
\keywords{Topological complexity, configuration spaces}
\subjclass{55M30, 55R80}
\maketitle

\begin{abstract}
    Given a space $X,$ the topological complexity of $X,$ denoted by $\TC(X)$, can be viewed as the minimum number of ``continuous rules" needed to describe how to move between any two points in $X.$ Given subspaces $Y_1$ and $Y_2$ of $X,$ there is a ``relative" version of topological complexity, denoted by $\TC_X(Y_1\times Y_2),$ in which one only considers paths starting at a point $y_1\in Y_1$ and ending at a point $y_2\in Y_2,$ but the path from $y_1$ to $y_2$ can pass through any point in $X.$ We discuss general results that provide relative analogues of well-known results concerning $\TC(X)$ before focusing on the case in which we have $Y_1=Y_2=C^n(Y),$ the configuration space of $n$ points in some space $Y,$ and $X=C^n(Y\times I),$ the configuration space of $n$ points in $Y\times I,$ where $I$ denotes the interval $[0,1].$ Our main result shows $\TC_{C^n(Y\times I)}(C^n(Y)\times C^n(Y))$ is bounded above by $\TC(Y^n)$ and under certain hypotheses is bounded below by $\TC(Y).$
\end{abstract}

\section{Introduction}\label{sec:Intro}

Given a topological space $X,$ let $P(X)=X^I,$ the space of all continuous maps $\sigma\from I\to X,$ where $I$ denotes the unit interval $[0,1].$ We equip $P(X)$ with the compact-open topology. There is a fibration $p\from P(X)\to X\times X$ which sends a path $\sigma$ to its endpoints:
\begin{equation}\label{eqn:TCFibration}
    p(\sigma)=(\sigma(0),\sigma(1)).
\end{equation}

A section of this fibration is a function $s\from X\times X\to P(X)$ such that $p\circ s$ is the identity on $X\times X.$ In other words, $s$ is a function which takes a pair of points in $X$ as input and produces a path between those points. Intuitively speaking, if $s$ is continuous at a point $(x,y)\in X\times X,$ then a slight perturbation of the point $(x,y)$ results in a slight perturbation of the path $s(x,y).$ One easily shows that a continuous section $s\from X\times X\to P(X)$ exists if and only if the space $X$ is contractible \cite{farber_2003}. This is the motivation for Farber's definition of the \textit{topological complexity} of $X.$

\begin{defn}{\rm \cite{farber_2003}}\label{defn:TC}
Given a space $X$, let $\TC(X)$ denote the smallest integer $k$ such that there exists an open cover of $X\times X$ by sets $U_1,\dots,U_k$ which admit continuous sections $s_i\from U_i\to P(X)$ for each $i.$ If no such $k$ exists, set $\TC(X)=\infty.$
\end{defn}

\begin{rmk}
Farber shows in \cite{farber_2008} that for the case in which $X$ is a Euclidean neighborhood retract, the open cover $U_1,\dots, U_k$ in Definition \ref{defn:TC} can be replaced with a \textit{partition} of $X\times X$ into sets $K_1,\dots,K_k$.
\end{rmk}

Many authors use a ``reduced" version of topological complexity, which is one less than the value of $\TC(X)$ given in Definition \ref{defn:TC}. We work exclusively with the ``unreduced" version given in Definition \ref{defn:TC}. The sections $s_i$ can be thought of as ``continuous rules" which specify how to move from a point $x$ to another point $y$ where $(x,y)\in U_i.$ The sets $U_i$ and sections $s_i$ are collectively referred to as a \textit{motion planning algorithm} in $X.$ The space $X$ is often interpreted as the space of all configurations of a robot or a system of robots, and the computation of $\TC(X)$ addresses the question of how to move from any initial configuration to any final configuration of the robot(s). However, one can imagine situations in which the set of all pairs of initial-final configurations of the robot(s) forms only a subset $A\subseteq X\times X,$ but it is desirable to allow the robot(s) to move anywhere in $X$ as they move from the initial configuration to the final configuration. This leads to the notion of \textit{relative} topological complexity, which is denoted by $\TC_X(A).$ Let $P_X(A)$ denote the subspace of $P(X)$ which consists of all paths in $X$ from a point $x$ to a point $y$ with $(x,y)\in A.$

\begin{defn}{\rm\cite{farber_2008}}\label{defn:RelativeTC}
Given a subspace $A\subseteq X\times X,$ let $\TC_X(A)$ denote the smallest integer $k$ such that there exists an open cover of $A$ by sets $U_1,\dots, U_k$ which admit continuous sections $s_i\from U_i\to P_X(A).$ If no such $k$ exists, set $\TC_X(A)=\infty.$
\end{defn}

\begin{rmk}
As with $\TC(X),$ if $X$ is a Euclidean neighborhood retract, the open cover in Definition \ref{defn:RelativeTC} can be replaced with a partition of $A.$
\end{rmk}

In \cite{short_2018}, Short introduces the relative topological complexity of a pair $(X,Y),$ which is denoted by $\TC(X,Y)$ and can be defined by $\TC(X,Y)=\TC_X(X\times Y).$ Thus, $\TC(X,Y)$ addresses the problem of finding rules to move from points in $X$ to points in $Y\subseteq X$ along paths in $X.$ We will be interested in $\TC_X(Y_1\times Y_2)$ for subsets $Y_1$ and $Y_2$ of $X.$ Thus, we address the problem of finding rules to move from points in $Y_1$ to points in $Y_2$ along paths in $X.$ In Section \ref{sec:GeneralResults} we discuss general results concerning $\TC_X(Y_1\times Y_2)$ which give analogues of well-known results regarding $\TC(X)$ and provide elementary examples. 

In Section \ref{sec:ConfigSpaces}, we discuss the case in which $Y_1=Y_2=C^n(Y),$ where, for any space $Y$, $C^n(Y)$ denotes the configuration space of $n$ points in $Y$. In other words, $C^n(Y)$ may be viewed as the space of all $n$-tuples of \textit{distinct} points in $Y$. We will view $Y$ as a subspace of some larger space $Y'$  and in turn view $C^n(Y)$ as a subspace of $C^n(Y').$ If we view each point in $C^n(Y')$ as a configuration of $n$ robots, we can interpret $\TC_{C^n(Y')}(C^n(Y)\times C^n(Y))$ as the minimum number of continuous rules needed to describe how to move all $n$ robots from any configuration in $Y$ to any other configuration in $Y$, while allowing the robots to move along paths through $Y'$ (while avoiding collisions). We will primarily be interested in the case in which $Y'$ is of the form $Y\times I,$ which contains $Y$ as the subspace $Y\times \{0\}.$ Let $X=C^n(Y\times I).$ In this context, the paths involved in $\TC_X(C^n(Y)\times C^n(Y))$ start and end in configurations in $Y$ but are allowed to move ``above" $Y$ during the intermediate stages. We show that under certain hypotheses, we have 
\[
\TC(Y)\le\TC_X(C^n(Y)\times C^n(Y))\le \TC(Y^n).
\]
The upper bound $\TC_X(C^n(Y)\times C^n(Y))\le \TC(Y^n)$ is true in general. We also show that this bound is sharp in some cases, but not in others.

\section{General Results}\label{sec:GeneralResults}

We first recall standard results regarding $\TC(X).$

\begin{thm}{\rm \cite{farber_2003,farber_2004}}\label{thm:StandardTC}
Let $X$ be any topological space.
\begin{enumerate}[(a)]
    \item\label{item:TC1} We have $\TC(X)=1$ if and only if $X$ is contractible.
    \item\label{item:HtpyInv} If $X$ is homotopy equivalent to some space $X'$, then $\TC(X)=\TC(X').$
    \item\label{item:TCUB} Let $X$ be a CW complex of dimension $n.$ If $\pi_j(X)=0$ for $j\le s$, then $\TC(X)< \frac{2n+1}{s+1}+1.$ In particular, if $X$ is path-connected, we have $\TC(X)\le 2n+1.$
    \item\label{item:TCLB} Let $\mathbf{k}$ be a field and consider elements 
    \[
    \alpha_1,\dots,\alpha_k\in \ker(\smallsmile\from H^*(X;\mathbf{k})\otimes H^*(X;\mathbf{k})\to H^*(X;\mathbf{k})).
    \]
    If $\alpha_1\cdots\alpha_k\ne 0\in H^*(X;\mathbf{k})\otimes H^*(X;\mathbf{k}),$ then $\TC(X)> k.$
\end{enumerate}
\end{thm}

The main goal of this section is to discuss analogues of these results for $\TC_X(Y_1\times Y_2).$ Our proofs are generalizations of the proofs of the results of Theorem \ref{thm:StandardTC}. These results can also be compared with the results in \cite{short_2018} which address the case in which $Y_1=X.$  

We first mention some obvious inequalities (see \cite{farber_2008}). If $Y_1$ and $Y_2$ are subspaces of $X,$ we have 
\begin{equation}
    \TC_X(Y_1\times Y_2)\le \TC(Y_1\cup Y_2).
\end{equation}
In particular, taking $Y_1=Y_2=Y\subseteq X,$ we have $\TC_X(Y\times Y)\le \TC(Y).$ Intuitively, this inequality comes from the fact that when determining $\TC_X(Y\times Y),$ we have more choices of paths between points in $Y$ than we do when determining $\TC(Y).$ Next, if $Y_1\subseteq Y_1'\subseteq X$ and $Y_2\subseteq Y_2'\subseteq X,$ we have 
\begin{equation}\label{eqn:SubspaceUB}
    \TC_X(Y_1\times Y_2)\le \TC_X(Y_1'\times Y_2').
\end{equation}
In particular, taking $Y_1'=Y_2'=X,$ we have $\TC_X(Y_1\times Y_2)\le \TC_X(X\times X)=\TC(X).$ Intuitively, this inequality comes from the fact that when determining $\TC_X(Y_1\times Y_2)$, we have fewer pairs of initial and terminal points to consider than we do in determining $\TC(X).$

Now we turn to proving analogues of the results of Theorem \ref{thm:StandardTC}. In our proofs, we make implicit use of the following standard facts:

\begin{lemma}
Let $X,\ Y,$ and $Z$ be any topological spaces, and equip $P(Y)$ and $P(Z)$ with the compact-open topology.
\begin{enumerate}
\item A function $f\from X\times I\to Y$ is continuous if and only if the function $F\from X\to P(Y)$ given by $F(x)(t)=f(x,t)$ is continuous.
\item Let $F_1$ and $F_2$ be continuous functions $X\to P(Y)$ and consider continuous functions $\phi_1\from J_1\to I$ and $\phi_2\from J_2\to I,$ where $J_1$ and $J_2$ are closed subsets of $I$ with $J_1\cup J_2=I.$ If $F_1(x)(\phi_1(t))=F_2(x)(\phi_2(t))$ for all $y$ and each $t\in J_1\cap J_2,$ the function $G\from X\to P(Y)$ given by 
\[
G(x)(t)=\begin{cases}
F_1(x)(\phi_1(x)),& t\in J_1\\
F_2(x)(\phi_2(x)),& t\in J_2
\end{cases}
\]
is continuous.
\item Let $s\from X\to P(Y)$ and $f\from Y\to Z$ be continuous. The function $s'\from X\to P(Z)$ given by $s'(x)(t)=f(s(x)(t))$ is continuous.
\end{enumerate}
\end{lemma}

Now, we address the cases in which $\TC_X(Y_1\times Y_2)=1,$ giving an analogue of Theorem \ref{thm:StandardTC}\ref{item:TC1}. We first note that any two spaces $Y_1$ and $Y_2$ (which need not be disjoint) can be embedded in a space $X$ such that $\TC_X(Y_1\times Y_2)=1.$ Indeed, by taking $X= C(Y_1\cup Y_2)$ (where $CY$ denotes the cone $(Y\times I)/(Y\times \{1\})$, which contains $Y$ as the subspace $Y\times \{0\}$), we have $\TC(X)=1,$ since $X$ is contractible, and the inequality $\TC_X(Y_1\times Y_2)\le \TC(X)$ shows $\TC_X(Y_1\times Y_2)=1.$ 

As another example, recall given two disjoint spaces $Y_1$ and $Y_2$, the join of $Y_1$ and $Y_2$ is denoted by $Y_1\ast Y_2$ and is formed as a quotient of $Y_1\times Y_2\times I$ under the identifications $(y_1,y_2,0)\sim (y_1,y_2',0)$ and $(y_1,y_2,1)\sim (y_1',y_2,1).$ The join $Y_1\ast Y_2$ contains $Y_1$ and $Y_2$ as the images under the quotient map of the subspaces $Y_1\times Y_2\times \{0\}$ and $Y_1\times Y_2\times \{1\},$ respectively. We denote the image of $(y_1,y_2,t)$ under the quotient map by $[y_1,y_2,t],$ so that for any $y_1\in Y_1$ and $y_2\in Y_2,$ we may identify $[y_1,y_2,0]$ with $y_1,$ and $[y_1,y_2,1]$ with $y_2.$ Given $(y_1,y_2)\in Y_1\times Y_2,$ define a path $s(y_1,y_2)\in P_{Y_1\ast Y_2}(Y_1\times Y_2)$ from $y_1$  to $y_2$ by $s(y_1,y_2)(t)=[y_1,y_2,t].$ This gives a continuous section of $p\from P_{Y_1\ast Y_2}(Y_1\times Y_2)\to Y_1\times Y_2,$ showing 
\[
\TC_{Y_1\ast Y_2}(Y_1\times Y_2)=1.
\]
Note that in this case, $Y_1\ast Y_2$ is not necessarily contractible, but the inclusions $Y_1\hookrightarrow Y_1\ast Y_2$ and $Y_2\hookrightarrow Y_1\ast Y_2$ are both nullhomotopic. Theorem \ref{thm:TC1} shows this property determines the cases in which $\TC_X(Y_1\times Y_2)=1.$ This can be compared with Farber's result in \cite{farber_2008} which shows that for a general subset $A\subseteq X\times X,$ we have $\TC_X(A)=1$ if and only if the projections $\pi_i\from A\to X$ are homotopic ($i=1,2$). 

\begin{thm}\label{thm:TC1}
Let $Y_1$ and $Y_2$ be subspaces of a path-connected space $X.$ We have 
\[
\TC_X(Y_1\times Y_2)=1
\]
if and only if the inclusions $Y_1\hookrightarrow X$ and $Y_2\hookrightarrow X$ are both nullhomotopic.
\end{thm}

\begin{proof}
If $\TC_X(Y_1 \times Y_2) = 1$, then there exists a continuous section $s\from Y_1 \times Y_2 \to P_X(Y_1 \times Y_2)$. Fix some point $y_2\in Y_2\subseteq X$ and consider the homotopy $h_t\from Y_1 \to X$ defined by
\[
h_t(y_1) = s(y_1, y_2)(t).
\]
For all $y_1 \in Y_1$ we have $h_0(y_1) = y_1$ and $h_1(y_1) = y_2$. Therefore, the inclusion map $Y_1 \hookrightarrow X$ is homotopic to the constant map $Y_1 \to X$ with $y_1 \mapsto y_2$ for all $y_1 \in Y_1$, so $Y_1 \hookrightarrow X$ is nullhomotopic. Next, we can consider the homotopy $h_t\from Y_2 \to X$ defined by
\[
h_t(y_2) = s(y_1, y_2)(1-t)
\]
where $y_1 \in Y_1 \subseteq X$. This similarly establishes a homotopy between the inclusion map $Y_2 \hookrightarrow X$ and the constant map $Y_2 \to X$ with $y_2 \mapsto y_1$, so $Y_2 \hookrightarrow X$ is also nullhomotopic.

Conversely, suppose that the inclusions $Y_1\hookrightarrow X$ and $Y_2\hookrightarrow X$ are both nullhomotopic. Then there exist homotopies $g_t\from Y_1 \to X$ and $h_t\from Y_2 \to X$ which satisfy $g_0(y_1) = y_1$, $g_1(y_1) = x_1$, $h_0(y_2) = y_2$, and $h_1(y_2) = x_2$ for any $y_1 \in Y_1$, $y_2 \in Y_2$, and fixed points $x_1, x_2 \in X$. We can construct a continuous section $s\from Y_1 \times Y_2 \to P_X(Y_1 \times Y_2)$ defined by 
\[
s(y_1, y_2)(t) = \begin{cases} 
g_{3t}(y_1), &0 \leq t \leq \frac{1}{3}\\
\sigma(3t-1), &\frac{1}{3} \leq t \leq \frac{2}{3}\\
h_{3-3t}(y_2), &\frac{2}{3} \leq t \leq 1
\end{cases},
\]
where $\sigma$ is a fixed path between $x_1$ and $x_2$ (which exists because $X$ is path-connected). This section takes the pair $(y_1, y_2)$ to the path in $P_X(Y_1, Y_2)$ that traverses from $y_1$ to $x_1$ along the homotopy $g_t$, takes the fixed path $\sigma$ from $x_1$ to $x_2$, and then finally traverses from $x_2$ to $y_2$ along the homotopy $h_t$ in reverse.
\end{proof}

For example, if $Y_1$ and $Y_2$ are any proper subspaces of the sphere $S^n$, we have 
\[
\TC_{S^n}(Y_1\times Y_2)=1.
\]
This should be compared with the fact that $\TC(S^n)=2$ if $n$ is odd and $\TC(S^n)=3$ if $n$ is even \cite{farber_2003}, and $\TC_{S^n}(Y\times S^n)=2$ where $Y$ is again a proper subset of $S^n$ \cite{short_2018}.

Next, we state the following version of homotopy invariance for $\TC_X(Y_1\times Y_2),$ giving an analogue of Theorem \ref{thm:StandardTC}\ref{item:HtpyInv}. This should be compared with a result from \cite{farber_2008} which says that if $A$ and $B$ are subsets of $X\times X$ such that $A$ can be deformed into $B$ inside of $X\times X,$ then $\TC_X(A)=\TC_X(B).$

\begin{thm}\label{thm:HtpyInv}
Let $Y_1$ and $Y_2$ be subspaces of $X$, let $Y_1'$ and $Y_2'$ be subspaces of $X'$, and suppose there are maps $f\from X\to X',\ f'\from X'\to X,\ \alpha_j\from Y_j\to Y_j'$ and $\alpha'_j\from Y_j'\to Y_j$ ($j=1,2$) such that the following diagrams commute up to homotopy for $j=1,2.$

\[
\begin{tikzcd}
Y_j \arrow[r, hook, "\iota_j"] & X \arrow[d, "f"]\\
Y_j' \arrow[u, "\alpha_j'"] \arrow[r, hook, "\iota'_j"] & X'
\end{tikzcd}
\hspace{5mm}
\begin{tikzcd}
Y_j \arrow[r, hook, "\iota_j"] \arrow[d,"\alpha_j"] & X\\ 
Y_j' \arrow[r, hook, "\iota_j'"] & X' \arrow[u, "f'"]
\end{tikzcd}
\]
Then, $\TC_{X'}(Y'_1\times Y'_2)=\TC_X(Y_1\times Y_2).$ In particular, this equality holds if $f$ is a homotopy equivalence which restricts to homotopy equivalences $Y_j\to Y_j'.$
\end{thm}

\begin{proof}
Suppose $\TC_X(Y_1\times Y_2)=k.$ Then, we can find open sets $U_1,\dots,U_k$ which cover $Y_1\times Y_2$ and sections $s_i\from U_i\to P_X(Y_1\times Y_2).$ For each $i,$ let $U'_i=(\alpha'_1\times\alpha'_2)^{-1}(U_i),$ so that the sets $U'_1,\dots,U_k'$ form an open cover of $Y_1'\times Y_2'.$ 

Now, fix $(y_1',y_2')\in U_i'\subseteq Y_1'\times Y_2'$ and let $y_j=\alpha_j'(y_j')$ (for $j=1,2$). By definition of $U_i',$ we have $(y_1,y_2)\in U_i,$ so $s_i(y_1,y_2)$ is a path in $X$ from $y_1$ to $y_2.$ Then, the map $\tilde{s}_i\from [0,1]\to X'$ given by $t\mapsto f(s_i(y_1,y_2)(t))$ is a path in $X'$ from $f(y_1)$ to $f(y_2).$

For $j=1,2,$ let $H_{j,t}\from Y_j'\to X'$ be a homotopy of maps which satisfies $H_{j,0}=\iota_j'$ and $H_{j,1}=f\circ\iota_j\circ\alpha_j'.$ Keeping $y_j'$ fixed, the map $\sigma_j\from [0,1]\to X'$ given by $t\mapsto H_{j,t}(y_j')$ is a path from $\iota_j'(y_j')=y_j'$ to $f(\iota_j(\alpha_j'(y_j')))=f(y_j).$ Define a path $s_i'(y'_1,y_2')\in P_{X'}(Y_1'\times Y_2')$ by
\[
s_i'(y_1',y_2')(t)=
\begin{cases}
\sigma_1(3t),&0\le t\le\frac{1}{3}\\
\tilde{s}_i(3t-1),&\frac{1}{3}\le t\le\frac{2}{3}\\
\sigma_2(3-3t),&\frac{2}{3}\le t\le 1
\end{cases}.
\]

Defining $s_i'(y_1',y_2')$ in this manner for each pair $(y_1',y_2')\in U_i'$ for $i=1,\dots,k$ shows that we have $\TC_{X'}(Y_1'\times Y_2')\le k=\TC_{X}(Y_1\times Y_2).$ 
A symmetric argument shows that we also have $\TC_{X}(Y_1\times Y_2)\le \TC_{X'}(Y_1'\times Y_2'),$ proving $\TC_{X}(Y_1\times Y_2)= \TC_{X'}(Y_1'\times Y_2').$
The final statement in the theorem follows by taking $\alpha_j=f|_{Y_j}$ and letting $f'$ and $\alpha_j'$ be a homotopy inverses for $f$ and $\alpha_j,$ respectively.
\end{proof}

Before discussing analogues of parts \ref{item:TCUB} and \ref{item:TCLB} of Theorem \ref{thm:StandardTC}, we briefly discuss the relationship between topological complexity and Schwarz's notion of the \textit{genus} of a fibration.

\begin{defn}{\rm \cite{schwarz_1965}}
Let $p\from E\to B$ be a fibration. The genus of $p$, denoted by $\genus(p),$ is the smallest integer $k$ such that there is a cover of $B$ by open sets $U_1,\dots, U_k$ which admit continuous sections $U_i\to E.$ If no such $k$ exists, let $\genus(p)=\infty.$
\end{defn}

From this, we see that $\TC(X)=\genus(p),$ where $p\from P(X)\to X\times X$ is the fibration given in (\ref{eqn:TCFibration}), and if $\iota_j$ denotes the inclusion $Y_j\hookrightarrow X$ ($j=1,2$), then $\TC_X(Y_1\times Y_2)=\genus(q),$ where $q$ is the pullback fibration of $p$ under $\iota_1\times\iota_2.$ Theorems \ref{thm:homotopybound} and \ref{thm:zerodivisors} give general upper and lower bounds on $\genus(p),$ which we use to give upper and lower bounds on $\TC_X(Y_1\times Y_2)$ in Theorems \ref{thm:RelTCUB} and \ref{thm:RelTCLB}.

\begin{thm}{\rm \cite{schwarz_1965}} \label{thm:homotopybound}
Let $B$ be a $CW$ complex, and let $p\from E \to B$ be a fibration with fiber $F$ such that $\pi_j(F) = 0$ for $j < s$. Then,
\[
\genus(p) < \frac{\dim(B)+1}{s+1} + 1.
\]
\end{thm}

\begin{thm}\label{thm:RelTCUB}
Consider a space $X$ with subspaces $Y_1$ and $Y_2$ which are CW complexes. If $\pi_j(X)=0$ for $j\le s,$ then 
\[
\TC_X(Y_1\times Y_2)<\frac{\dim(Y_1)+\dim(Y_2)+1}{s+1}+1.
\]
\end{thm}

\begin{proof}
The fiber of the fibration $p\from P(X) \to X \times X$ defining $\TC(X)$ is $\Omega X$, the loop space of $X$. Since $q\from P_X(Y_1 \times Y_2) \to Y_1 \times Y_2$ is a pullback of $p$, it has the same fiber. Therefore, since $\pi_k(X) \simeq \pi_{k-1}(\Omega X)$, we have by Theorem \ref{thm:homotopybound} that if $\pi_j(X) \simeq \pi_{j-1}(\Omega X) = 0$ for $j < s-1$ (and hence $j \leq s$), then
\[
\TC_X(Y_1\times Y_2)=\genus(q)<\frac{\dim(Y_1 \times Y_2)+1}{s+1}+1 = \frac{\dim(Y_1)+\dim(Y_2)+1}{s+1}+1,
\]
as desired.
\end{proof}

\begin{rmk}
If $Y$ is an $n$-dimensional subspace of a path-connected space $X,$ Theorem \ref{thm:RelTCUB} gives 
\[
\TC_X(Y\times Y)\le 2n+1,
\]
which agrees with the upper bound on $\TC(Y)$ given in Theorem \ref{thm:StandardTC}\ref{item:TCUB} if $Y$ is also path-connected. However, if $X$ is a highly-connected space, the upper bounds in Theorem \ref{thm:RelTCUB} can give better upper bounds on $\TC_X(Y_1\times Y_2)$ than the upper bounds on $\TC(Y).$
\end{rmk}

\begin{thm}{\rm\cite{schwarz_1965}}\label{thm:zerodivisors}
Let $p\from E \to B$ be a fibration. If there exist cohomology classes $\zeta_1, \dots, \zeta_j \in H^*(B)$ for which $p^* \zeta_i = 0$ for each $i$ and $\zeta_1 \smile \dots \smile \zeta_j \neq 0$, then $\genus(p) > j$.
\end{thm}

\begin{thm}\label{thm:RelTCLB}
Let $\mathbf{k}$ be a field and suppose there are classes 
\[
\alpha_1,\dots,\alpha_j\in \ker(\smallsmile \from H^*(X;\mathbf{k})\otimes H^*(X;\mathbf{k})\to H^*(X;\mathbf{k}))
\]
such that $(\iota_1^*\otimes \iota_2^*)(\alpha_1\cdots\alpha_j)\ne0,$ where $\iota_1$ and $\iota_2$ are the inclusions of $Y_1$ and $Y_2$ into $X.$ Then $\TC_X(Y_1\times Y_2)>j.$
\end{thm}

\begin{proof}
Let $p\from P(X) \to X \times X$ be the fibration defining $\TC(X)$ and $q$ be the pullback under $\iota_1\times \iota_2,$ which defines $\TC_X(Y_1 \times Y_2)$. Consider the commutative diagrams illustrating these fibrations and the corresponding induced maps in cohomology, taking coefficients in the field $\mathbf{k}$ and identifying $H^*(X\times X)$ (resp. $H^*(Y_1\times Y_2)$) with $H^*(X)\otimes H^*(X)$ (resp. $H^*(Y_1)\otimes H^*(Y_2)$) under the K\"{u}nneth formula isomorphism.
\[
    \begin{tikzcd}
    P_X(Y_1 \times Y_2)\arrow[r,hook,"\iota"]\arrow[d,"q"]&P(X)\arrow[d,"p"]\\
    Y_1 \times Y_2\arrow[r,hook,"\iota_1 \times \iota_2"]&X\times X
    \end{tikzcd}
    \hspace{5mm}
    \begin{tikzcd}
    H^*(P_X(Y_1\times Y_2))&H^*(P(X))\arrow[l,"\iota^*"]\\
    H^*(Y_1) \otimes H^*(Y_2)\arrow[u,"q^*"]&H^*(X)\otimes H^*(X)\arrow[l,"\iota^*_1 \otimes \iota^*_2"]\arrow[u,"p^*"]
    \end{tikzcd}
\]
There exists a homotopy equivalence $c\from X \to P(X)$ sending a point $x$ to the constant path at $x$, so $H^*(P(X))$ is isomorphic to $H^*(X)$. Considering this map $c$ along with the fibration $p$ and the diagonal map $\Delta_X$ (mapping $x \mapsto (x,x)$), we have the following commutative diagrams (again identifying $H^*(X\times X)$ with $H^*(X)\otimes H^*(X)$).
\[
\begin{tikzcd}
X\arrow[r,"c"]\arrow[rd,swap,"\Delta_X"]&P(X)\arrow[d,"p"]\\&X\times X 
\end{tikzcd}
\begin{tikzcd}
H^*(X)&H^*(P(X))\arrow{l}{c^*}[swap]{\approx}\\
&H^*(X)\otimes H^*(X)\arrow{ul}{\Delta_X^*}\arrow[u,"p^*"]
\end{tikzcd}
\]
Since $\Delta_X$ induces the cup product $\smallsmile$, we have 
\[
\ker(\smallsmile\from H^*(X)\otimes H^*(X)\to H^*(X))=\ker(p^*\from H^*(X)\otimes H^*(X)\to H^*(P(X))).
\]

Now, suppose there exist elements $\alpha_1, \dots, \alpha_j \in \ker(p^*)$ such that $(\iota_1^*\otimes \iota_2^*)(\alpha_1\cdots\alpha_j)\ne 0.$ Let $\zeta_i=(\iota_1^*\otimes \iota_2^*)(\alpha_i)\in H^*(Y_1)\otimes H^*(Y_2).$  The commutativity of the initial diagram in cohomology shows $q^*(\zeta_i)=\iota^*(p^*(\alpha_i))=\iota^*(0)=0,$ so $\zeta_i\in \ker(q^*)$ for each $i.$ Thus by Theorem \ref{thm:zerodivisors} and the definition of $\TC_X(Y_1\times Y_2)$, we have that $\TC_X(Y_1 \times Y_2) > j$.
\end{proof}

As an example, consider the torus $T^2$ as the quotient space of $I\times I$ formed by identifying $(0,t)$ with $(1,t)$ and $(s,0)$ with $(s,1),$ and let $Y\subseteq T^2$ denote the subspace corresponding to the boundary of $I\times I$ (so that $Y$ is homeomorphic to a wedge of two circles). Let $\alpha$ and $\beta$ denote the generators of $H^*(T^2;\Z/2\Z),$ let 
\[
\overline{\alpha}=\alpha\otimes1+1\otimes\alpha\in H^*(T^2;\Z/2\Z)\otimes H^*(T^2;\Z/2\Z),
\]
and define $\overline{\beta}$ similarly. We have 
\[
\overline{\alpha},\overline{\beta}\in \ker(\smallsmile\from H^*(T^2;\Z/2\Z)\otimes H^*(T^2;\Z/2\Z)\to H^*(T^2;\Z/2\Z)),
\]
and if $A=\iota^*(\alpha)$ and $B=\iota^*(\beta)$ (where $\iota\from Y\hookrightarrow T^2$ is the inclusion), we have 
\[
(\iota^*\otimes\iota^*)(\overline{\alpha}\overline{\beta})=AB\otimes 1+A\otimes B+B\otimes A+1\otimes AB=A\otimes B+B\otimes A\ne0.
\]
So, Theorem \ref{thm:RelTCLB}, together with the well-known facts that $\TC(T^2)=\TC(S^1\vee S^1)=3$ and the upper bounds discussed at the beginning of this section give
\[
\TC_{T^2}(Y\times Y)=\TC(Y)=\TC(T^2)=3.
\]

Finally, we briefly mention the relationship between the  topological complexity of $X$ and the Lusternik-Schnirelmann category of $X$. Given a space $X,$ recall the Lusternik-Schnirelmann category is denoted by $\cat(X)$ and can be defined as the smallest integer $k$ such that $X$ can be covered by $k$ open sets $W_1,\dots,W_k$ such that each inclusion $W_i\hookrightarrow X$ is nullhomotopic. If no such $k$ exists, set $\cat(X)=\infty.$ We have the following relationships between the topological complexity of $X,$ the relative topological complexity of a pair $(X,Y)$ and the Lusternik-Schnirelmann categories of $X$ and $X\times X$ \cite{farber_2003,short_2018}:

\[
\cat(X)\le\TC(X,Y)=\TC_X(X\times Y)\le\TC(X)\le\cat(X \times X).
\]
The inequality $\TC_X(Y_1\times Y_2)\le\min\{\TC(X),\TC(Y_1\cup Y_2)\}$ shows that we have 
\[
\TC_X(Y_1\times Y_2)\le\min\{\cat(X\times X),\cat((Y_1\cup Y_2)\times(Y_1\cup Y_2)\}.
\]
However, $\TC_X(Y_1\times Y_2)$ need not be bounded below by $\cat(X)$ or $\cat(Y_1\cup Y_2).$ For example, if $Y_1=Y_2=Y$ is a contractible subspace of a non-contractible space $X,$ we have 
\[
\TC_X(Y\times Y)\le\TC(Y)=1<\cat(X),
\]
and on the other hand, if $Y_1=Y_2=Y$ is a non-contractible subspace of a contractible space $X,$ we have 
\[
\TC_X(Y\times Y)\le\TC(X)=1<\cat(Y).
\]
Instead, we can work with a relative version of Lusternik-Schnirelmann category for a subspace $Y\subseteq X,$ which we denote by $\cat_X(Y),$ and is defined in the same way as $\cat(Y),$ with the exception that we only require that each inclusion $W_i\hookrightarrow X$ (rather than $W_i\hookrightarrow Y$) be nullhomotopic. With this, one easily shows that we have 
\[
\max\{\cat_X(Y_1),\cat_X(Y_2)\}\le\TC_X(Y_1\times Y_2)\le\cat_{X\times X}(Y_1\times Y_2).
\]

\section{Configuration Spaces}\label{sec:ConfigSpaces}

We now turn our attention to configuration spaces. For any space $Y,$ recall the configuration space of $n$ points in $Y$ is the subspace of $Y^n$ consisting of $n$-tuples of distinct points. In other words, 
\[
C^n(Y)=\{(y_1,\dots,y_n)\in Y^n|y_i\ne y_j\text{ for }i\ne j\}.
\]
As mentioned in Section \ref{sec:Intro}, we will view $Y$ as a subspace of some larger space $Y'$ and consider $\TC_{C^n(Y')}(C^n(Y)\times C^n(Y)).$ Viewing elements of the configuration spaces as configurations of $n$ robots, we can interpret the space $Y$ as the space of locations at which the robots are required to perform tasks, and the space $Y'$ as the space throughout which the robots may move. As an elementary example, if $Y$ is a discrete subspace of a space $Y'$ such that $Y$ has at least $n$ points and $C^n(Y')$ is path-connected, then the inclusion $C^n(Y)\hookrightarrow C^n(Y')$ is nullhomotopic, so Theorem \ref{thm:TC1} shows $\TC_{C^n(Y')}(C^n(Y)\times C^n(Y))=1$ (or simply note that any function $C^n(Y)\times C^n(Y)\to P_{C^n(Y')}(C^n(Y)\times C^n(Y))$ is continuous, and such a function exists since $C^n(Y')$ is path-connected).

We also note that given any space $Y$ with at least $n$ points, we can embed $Y$ in a larger space $Y'$ such that $\TC_{C^n(Y')}(C^n(Y)\times C^n(Y))=1.$ Indeed, given a space $Y$ with at least $n$ points, let $Z$ be a discrete space with $n$ points $z_1,\dots,z_n$ and let $Y'=Y\ast Z.$ The space $Y'$ can also be interpreted as $n$ copies of the cone $CY$ identified along $Y$. Again viewing $Y$ as the subspace of $Y\ast Z$ which consists of all points of the form $[y,z_i,0],$ define $H_t\from C^n(Y)\to C^n(Y')$ by 
$H_t(y_1,\dots,y_n)=([y_1,z_1,t],\dots,[y_n,z_n,t]).$ This gives a homotopy from the inclusion $C^n(Y)\hookrightarrow C^n(Y')$ to the constant map to the point $([y_1,z_1,1],\dots,[y_n,z_n,1])$ in $C^n(Y'),$ so Theorem \ref{thm:TC1} shows $\TC_{C^n(Y')}(C^n(Y)\times C^n(Y))=1.$

Our primary interest is the case in which we view $Y$ as the subspace $Y\times \{0\}\subseteq Y\times I.$ For the remainder of this section, $Y$ is any space with at least $n$ points, and we denote the space $C^n(Y\times I)$ by $X$.

\begin{lemma}\label{lemma:ConfigVsProduct}
We have $\TC_{X}(C^n(Y)\times C^n(Y))=\TC_{Y^n}(C^n(Y)\times C^n(Y)).$
\end{lemma}

\begin{proof}
First, suppose $\TC_{Y^n}(C^n(Y)\times C^n(Y))=k,$ and let $U_1,\dots,U_k$ form an open cover of $C^n(Y)\times C^n(Y)$ which admits sections $s_i\from U_i\to P_{Y^n}(C^n(Y)\times C^n(Y)).$ We wish to define a section $\tilde{s}_i\from U_i\to P_{X}(C^n(Y)\times C^n(Y))$ for $i=1,\dots,k.$

For each $\vec{y}=(y_1,\dots,y_n)\in C^n(Y),$ let $l_{\vec{y}}=(l_1,\dots,l_n)\from I\to X$ be the path given by $l_j(t)=(y_j,\frac{t}{j}).$ Let $f\from Y^n\to X$ be given by $f(y_1,\dots,y_n)=((y_1,1),\dots,(y_n,\frac{1}{n})).$ Finally, define $\tilde{s}_i\from U_i\to P_X(C^n(Y)\times C^n(Y))$ by 
\[
\tilde{s}_i(\vec{x},\vec{y})(t)=
\begin{cases}
l_{\vec{x}}(3t),&0\le t\le \frac{1}{3}\\
f(s_i(\vec{x},\vec{y})(3t-1)),&\frac{1}{3}\le t\le\frac{2}{3}\\
l_{\vec{y}}(3-3t),&\frac{2}{3}\le t\le1
\end{cases}.
\]
In other words, $\tilde{s}_i$ lifts robot $j$ from ``height" 0 to height $\frac{1}{j},$ then follows the path given by $s_i,$ (at height $\frac{1}{j}$), then moves robot $j$ back down to height 0. Each $\tilde{s}_i$ is continuous, showing $\TC_X(C^n(Y)\times C^n(Y))\le \TC_{Y^n}(C^n(Y)\times C^n(Y)).$

Conversely, suppose $\TC_X(C^n(Y)\times C^n(Y))=k,$ so that we may find an open cover $U_1,\dots, U_k$ of $C^n(Y)\times C^n(Y)$ and sections $\tilde{s}_i\from U_i\to P_X(C^n(Y)\times C^n(Y)).$ With this, we define $s_i\from U_i\to P_{Y^n}(C^n(Y)\times C^n(Y))$ by 
\[
s_i(\vec{x},\vec{y})(t)= g(\tilde{s}_i(\vec{x},\vec{y})(t)),
\]
where $g\from X\to Y^n$ is given by $((y_1,t_1),\dots,(y_n,t_n))\mapsto (y_1,\dots,y_n).$ Each $s_i$ is continuous, showing $\TC_{Y^n}(C^n(Y)\times C^n(Y))\le \TC_X(C^n(Y)\times C^n(Y)),$ completing the proof.
\end{proof}

\begin{thm}\label{thm:ConfigUB}
We have $\TC_X(C^n(Y)\times C^n(Y))\le \TC(Y^n).$
\end{thm}

\begin{proof}
From Lemma \ref{lemma:ConfigVsProduct}, we have $\TC_{X}(C^n(Y)\times C^n(Y))=\TC_{Y^n}(C^n(Y)\times C^n(Y)),$ and from the inequality given in \eqref{eqn:SubspaceUB}, we have 
\[
\TC_{Y^n}(C^n(Y)\times C^n(Y))\le\TC_{Y^n}(Y^n\times Y^n)=\TC(Y^n).
\]
\end{proof}
\begin{rmk} 
For the case in which $Y$ is a Euclidean neighborhood retract, Farber shows that we have $\TC(Y^n)\le n\cdot\TC(Y)-n+1$ \cite{farber_2008}, so in this case, Theorem \ref{thm:RelTCUB} gives upper bounds on $\TC_X(C^n(Y)\times C^n(Y))$ in terms of $\TC(Y)$ (and $n$).
\end{rmk}

\begin{cor}
If $Y$ is contractible, we have $\TC_X(C^n(Y)\times C^n(Y))=1.$
\end{cor}

\begin{proof}
This follows immediately from Theorem \ref{thm:ConfigUB} and the fact that $Y^n$ is contractible whenever $Y$ is contractible, so $\TC(Y^n)=1$ by Theorem \ref{thm:StandardTC}\ref{item:TC1}.
\end{proof}

It is worth comparing this with results regarding $\TC(C^n(Y))$, which is rarely trivial even when $Y$ is contractible. For example, in the case of Euclidean space $\R^m,$ for $n\ge 2$ one has $\TC(C^n(\R^m))=2n-\epsilon$ where $\epsilon=1$ for $m\ge 3$ odd and $\epsilon=2$ for $m\ge 2$ even \cite{farber_yuzvinsky_2004,farber_grant_2008}.  In the case of a tree $T$ (i.e. a simply connected 1-dimensional CW complex) which is homeomorphic to neither $S^1$ nor $I$, one has $\TC(C^n(T))=2\min\{m,\lfloor n/2\rfloor\}+1,$ where $m$ is the number of vertices of degree greater than 2 in $T.$ \cite{farber_2017,lutgehetmann_recio-mitter_2019}.

Now, we use Theorem \ref{thm:RelTCLB} to show the upper bound given in Theorem \ref{thm:ConfigUB} is sharp for the case in which $Y$ is a connected graph (i.e. a connected 1-dimensional CW complex) with at least $n$ disjoint cycles. Recall a cycle in a graph $Y$ is a subset $\mathcal{C}$ which is homeomorphic to $S^1.$ In the proof, we take coefficients in $\Z/2\Z$ and omit the coefficients from our notation.

\begin{prop}
Let $Y$ be a connected graph with at least $n$ cycles $\mathcal{C}_1,\dots,\mathcal{C}_n$ satisfying $\mathcal{C}_i\cap \mathcal{C}_j=\emptyset$ for $i\ne j.$ Then, 
\[
\TC_X(C^n(Y)\times C^n(Y))= \TC(C^n(Y))=\TC(Y^n)=2n+1.
\]
\end{prop}

\begin{proof}
In \cite{scheirer_2020}, it is shown that under these hypotheses, we have $\TC(C^n(Y))=2n+1.$ The proof involves the construction of classes $\mu_j$ and $\mu_j'$ in $H^*(C^n(Y))$ for $j=1,\dots,n$ which we recall here. For each $j=1,\dots,n,$ let $b_j\in H_*(Y)$ denote a homology class corresponding to the cycle $\mathcal{C}_j,$ and let $\beta_j\in H^*(Y)$ denote the dual class. Then, define $\gamma_j,\gamma'_j\in H^*(Y^n)$ by
\[
\gamma_j=1\times\cdots\times \beta_j\times\cdots\times1,
\]
where the only non-trivial term, $\beta_j$, falls in the $j$th factor, and 
\[
\gamma'_j=1\times\cdots\times \beta_{j+1}\times\cdots\times1,
\]
where the only non-trivial term, $\beta_{j+1}$, again falls in the $j$th factor (subscripts to be read modulo $n$). Let $\iota\from C^n(Y)\hookrightarrow Y^n$ denote the inclusion, and let $\mu_j=\iota^*(\gamma_j)$ and $\mu_j'=\iota^*(\gamma_j').$
It is shown in \cite{scheirer_2020} that the product 
\[
\biggl(\prod_{j=1}^n(\mu_j\otimes1+1\otimes\mu_j)\biggr)\cdot\biggl(\prod_{j=1}^n(\mu'_j\otimes1+1\otimes\mu'_j)\biggr)
\]
is nonzero. However, this product is the image under $\iota^*\otimes \iota^*$ of $\alpha_1\cdots\alpha_n\alpha_1'\cdots\alpha_n',$ where $\alpha_j=\gamma_j\otimes1+1\otimes\gamma_j,$ and $\alpha_j'=\gamma_j'\otimes1+1\otimes\gamma_j'.$ We have 
\[\alpha_j,\alpha_j'\in\ker(\smallsmile\from H^*(Y^n)\otimes H^*(Y^n)\to H^*(Y^n)),
\]
and since $(\iota^*\otimes \iota^*)(\alpha_1\cdots\alpha_n\alpha'_1\cdots\alpha'_n)\ne0,$ Theorem \ref{thm:RelTCLB} shows 
\[
\TC_{Y^n}(C^n(Y)\times C^n(Y))\ge2n+1.
\]
Since $Y^n$ is an $n$-dimensional complex, we have $\TC(Y^n)\le 2n+1,$ giving
\[
2n+1\le\TC_{Y^n}(C^n(Y)\times C^n(Y))\le\TC(Y^n)\le 2n+1,
\]
so $\TC_X(C^n(Y)\times C^n(Y))=\TC_{Y^n}(C^n(Y)\times C^n(Y))=\TC(Y^n)=2n+1,$ completing the proof.
\end{proof}

Next, we show that the upper bound in Theorem \ref{thm:ConfigUB} is not sharp in general. Recall that if $m\ge 2$ is even, we have $\TC((S^m)^n)=2n+1$ \cite{farber_2003}.

\begin{prop}
If $m\ge 2$ is an even integer and $X=C^n(S^m\times I)$, we have 
\[
\TC_X(C^n(S^m)\times C^n(S^m))\le n+2.
\]
\end{prop}

\begin{proof}
Our proof is a modification of Farber's proof that $\TC((S^m)^n)=2n+1$ in \cite{farber_2003}. According to Lemma \ref{lemma:ConfigVsProduct}, it suffices to show $\TC_{(S^m)^n}(C^n(S^m)\times C^n(S^m))\le n+2.$ We show that we can partition $C^n(S^m)\times C^n(S^m)$ into $n+2$ sets, each of which admits a continuous section of the fibration $P_{(S^m)^n}(C^n(S^m)\times C^n(S^m))\to C^n(S^m)\times C^n(S^m)$ (see the remark following Definition \ref{defn:RelativeTC}).

Fix some $z\in S^m$ and let $V\from S^m-\{z\}\to S^m$ be a unit tangent vector field (i.e. $|V(x)|=1$ and $x\cdot V(x)=0$ for each $x\in S^m-\{z\}$).

Let $K_1=\{(x,y)|x\ne-y\}\subseteq S^m\times S^m,$ and define $s_1\from U_1\to P(S^m)$ by letting $s_1(x,y)$ be the path in $ S^m$ which travels along the shortest path from $x$ to $y$ at constant speed.

Let $K_2=\{(x,-x)|x\ne z\}\subseteq S^m\times S^m,$ and define $s_2\from U_2\to P(S^m)$ by letting $s_2(x,-x)$ be the path in $S^m$ given by $s_2(x,-x)(t)= x\cdot\cos(\pi t)+V(x)\cdot\sin(\pi t).$ 

Finally, let $K_3=\{(z,-z)\}\subseteq S^m\times S^m,$ and let $s_3\from K_3\to P(S^m)$ map $(z,-z)$ to any fixed path from $-z$ to $z.$ Note the sets $K_i$ and sections $s_i$ show $\TC(S^m)\le 3$ (which is easily shown to be a sharp upper bound).

Now, for $J=(j_1,\dots,j_n)\in \{1,2,3\}^n,$ let $L_J=\phi(K_{j_1}\times\cdots\times K_{j_n})\subseteq (S^m)^n\times (S^m)^n,$ where $\phi$ is the obvious homeomorphism $(S^m\times S^m)^n\to (S^m)^n\times (S^m)^n.$ Define $S_J\from L_J\to P((S^m)^n)$ by
\[
S_J((x_1,\dots,x_n),(y_1,\dots,y_n))(t)=(s_{j_1}(x_1,y_1)(t),\dots,s_{j_n}(x_n,y_n)(t)).
\]
Each $S_J$ is continuous. Now, for such a $J$ let $|J|=j_1+\dots+j_n.$  If $J'\in \{1,2,3\}^n$ satisfies $|J'|=|J|$ but $J'\ne J,$ then $\overline{L}_J\cap L_{J'}=\emptyset.$ Indeed, suppose $\{(x^t,y^t)\}_{t=1}^\infty$ is a sequence of points in $K_j$ which converges to some point $(x,y).$ By examining the definitions of each $K_j,$ we see that $(x,y)$ must fall in some $K_{j'}$ for $j'\ge j.$ Therefore, if 
$\{((x_1^t,\dots,x_n^t),(y_1^t,\dots,y_n^t))\}_{t=1}^\infty$ is a sequence of points in $L_J$ which converges to a point $((x_1,y_1),\dots,(x_n,y_n))$ in $L_{J'}$ for $J'=(j'_1,\dots,j'_n),$ we have $j'_k\ge j_k$ for $k=1,\dots,n.$ So, if $|J'|=|J|,$ we must have $j'_k=j_k$ for each $k,$ in which case $J'=J.$ 

So letting $L_j$ denote the union of all $L_J$ such that $|J|=j,$ we get a continuous function $S_j\from L_j\to P((S^m)^n).$ The sets $L_j$ for $j=n,n+1,\dots,3n$ partition $(S^m)^n\times (S^m)^n,$ which shows $\TC((S^m)^n)\le 2n+1.$

However, given any $((x_1,\dots,x_n),(y_1,\dots,y_n))\in C^n(Y)\times C^n(Y),$ there is at most one $j$ such that $x_j=z,$ so $C^n(Y)\times C^n(Y)$ is covered by sets of the form $L_J,$ where 3 appears at most once in $J.$
Therefore, we only need the sets $L_j$ for $j=n,\dots,2n+1$ to cover $C^n(S^m)\times C^n(S^m),$ showing $\TC_{(S^m)^n}(C^n(S^m)\times C^n(S^m))\le n+2.$
\end{proof}

In general, in order to use Theorem \ref{thm:RelTCLB} to find lower bounds on $\TC_X(C^n(Y)\times C^n(Y)),$ we need an understanding of the map in cohomology induced by $C^n(Y)\hookrightarrow C^n(Y\times I)$ (or $C^n(Y)\hookrightarrow Y^n)$. However, in Theorems \ref{thm:FixedPointFree} and \ref{thm:HomeoRetract} we show under certain hypotheses $\TC_X(C^n(Y)\times C^n(Y))$ is bounded below by $\TC(Y)$. Before stating and proving these results, we note that at first glance, this may seem like an obvious lower bound, but it is not always obvious how to use a motion planning algorithm for $\TC_X(C^n(Y)\times C^n(Y)),$ or, equivalently, a motion planning algorithm for $\TC_{Y^n}(C^n(Y)\times C^n(Y)),$ to obtain one for $\TC(Y).$ For example, consider the case in which $Y=\R$ and $n=2.$ Define a section $\sigma_0\from \R\times \R\to P(\R)$ by 
\[
\sigma_0(x,y)(t)=
\begin{cases}
(1-2t)\cdot x,&0\le t\le \frac{1}{2}\\
(2t-1)\cdot y,&\frac{1}{2}\le t\le 1
\end{cases},
\]
so that $\sigma_0(x,y)$ moves from $x$ to 0, then from 0 to $y.$ Similarly, define $\sigma_1\from \R\times \R\to P(\R)$ by
\[
\sigma_1(x,y)(t)=\begin{cases}
(1-2t)\cdot x+2t,&0\le t\le \frac{1}{2}\\
(2t-1)\cdot y+2-2t,&\frac{1}{2}\le t\le 1
\end{cases},
\]
so that $\sigma_1(x,y)$ moves from $x$ to 1, then from 1 to $y.$

Now, define a section $\sigma\from C^2(\R)\times C^2(\R)\to P_{\R^2}(C^2(\R)\times C^2(\R))$ by 
\[
\sigma((x_1,x_2),(y_1,y_2))(t)=\begin{cases}
(\sigma_0(x_1,y_1)(t),\sigma_0(x_2,y_2)(t)),&\text{if }x_1< x_2\\
(\sigma_1(x_1,y_1)(t),\sigma_1(x_2,y_2)(t)),&\text{if }x_1> x_2
\end{cases}.
\]
This shows $\TC_X(C^2(\R)\times C^2(\R))=\TC_{\R^2}(C^2(\R)\times C^2(\R))=1$ (where $X=C^2(\R\times I)$).

Of course we also have $\TC(\R)=1,$ but we wish to consider how we can show this by using $\sigma$ to obtain a section $s\from \R\times \R\to P(\R).$ Perhaps the most obvious way to attempt this is to fix some point $z\in \R$ and consider the subspace $\widetilde{Y}=\{(x,z)|x\ne z\}\subseteq C^2(\R).$ If $p_1\from\R^2\to \R$ is the projection of the first factor, then we have $p_1(\widetilde{Y})=\R-\{z\}$ and a section $s'\from (\R-\{z\})\times (\R-\{z\})\to P(\R)$ given by 
\[
s'(x,y)(t)=p_1(\sigma((x,z),(y,z))(t)).
\]
However, it is not possible to extend this to a continuous section $s\from \R\times \R\to P(\R).$

Instead, we can consider, for example, the subspace $\widetilde{Y}=\{(x,x+1)|x\in \R\}\subseteq C^2(\R).$ In this case, we have $p_1(\widetilde{Y})=\R,$ so we can use $p_1$ to get a continuous section $s\from \R\times \R\to P(\R)$ given by 
\[
s(x,y)(t)=p_1(\sigma((x,x+1),(y,y+1))(t)).
\]
This is the motivation behind Theorem \ref{thm:FixedPointFree}.

Alternatively, we can consider, for example, the subspace $\widetilde{Y}=(0,1)\times \{2\}\subseteq C^2(\R),$ and use $p_1$ to get a continuous section $s'\from (0,1)\times (0,1)\to P_{\R}((0,1)\times (0,1))$ given by 
\[
s'(x,y)(t)=p_1(\sigma((x,2),(y,2))(t)),
\]
and then use a retraction $\R\to (0,1)$ to get a section $s''\from (0,1)\times (0,1)\to P((0,1)),$ which corresponds to a section $s\from \R\times \R\to P(\R)$ under the homeomorphism $(0,1)\approx \R.$ This is the motivation behind Theorem \ref{thm:HomeoRetract}.

\begin{thm}\label{thm:FixedPointFree}
Suppose $Y$ admits $n-1$ fixed-point-free maps $f_i\from Y\to Y$ $(i=1,2,\dots,n-1)$ satisfying $f_i(y)\ne f_j(y)$ for $i\ne j.$ Then, $\TC_X(C^n(Y)\times C^n(Y))\ge\TC(Y).$
\end{thm}

\begin{proof}
 Consider the subspace $\widetilde{Y}\subseteq Y^n$ consisting of all points of the form 
 \[
 \vec{y}=(y,f_1(y),f_2(y),\dots,f_{n-1}(y)).
 \]
 By assumption, we have $\widetilde{Y}\subseteq C^n(Y),$ so the inequality given in \eqref{eqn:SubspaceUB} shows 
 \[
 \TC_X(\widetilde{Y}\times \widetilde{Y})\le \TC_X(C^n(Y)\times C^n(Y)).
 \]
 Suppose $\TC_X(\widetilde{Y}\times \widetilde{Y})=k,$ and let $\widetilde{U}_1,\dots, \widetilde{U}_k$ form an open cover of $\widetilde{Y}\times\widetilde{Y}$ which admits sections $\tilde{s}_i\from \widetilde{U}_i\to P_X(\widetilde{Y}\times\widetilde{Y}).$
 
 Let $U_i=(p_1\times p_1)(\widetilde{U}_i)\subseteq Y\times Y,$ where $p_1\from Y^n\to Y$ is the projection of the first factor. Since $p_1$ gives a homeomorphism $\widetilde{Y}\to Y,$ the sets $U_1,\dots, U_k$ form an open cover of $Y\times Y.$

For each $(x,y)\in U_i\subseteq Y\times Y,$ the point $(\vec{x},\vec{y})\in \widetilde{Y}\times \widetilde{Y}\subseteq C^n(Y)\times C^n(Y)$ falls in $\widetilde{U}_i,$ so we have a continuous map $U_i\to\widetilde{U}_i$ given by $(x,y)\mapsto (\vec{x},\vec{y}).$
Define $F\from X\to Y$ by
\[
F((y_1,t_1),\dots,(y_n,t_n))=y_1,
\]
Finally, define $s_i\from U_i\to P(Y)$ by 
\[
s_i(x,y)(t)=F(\widetilde{s}_i(\vec{x},\vec{y})(t)).
\]
The function $s_i$ is continuous for each $i=1,\dots,k$ showing
\[
\TC(Y)\le k=\TC_{X}(\widetilde{Y}\times \widetilde{Y})\le\TC_X(C^n(Y)\times C^n(Y)).
\]

\end{proof}

\begin{thm}\label{thm:HomeoRetract}
Suppose $Y$ can be embedded in a space $Y'$ which is homeomorphic to $Y$ and admits a retraction $r\from Y'\to Y.$ If $Y'-Y$ has at least $n-1$ points, then 
\[
\TC_X(C^n(Y)\times C^n(Y))\ge\TC(Y).
\]
\end{thm}

\begin{proof}
Let $X'=C^n(Y'\times I),$ so that we may view $C^n(Y')$ as a subspace of $X'.$ We have a homeomorphism $X\approx X'$ which restricts to a homeomorphism $C^n(Y)\approx C^n(Y').$ Therefore, Theorem \ref{thm:HtpyInv} shows $\TC_X(C^n(Y)\times C^n(Y))= \TC_{X'}(C^n(Y')\times C^n(Y')).$

Now, let $z_2,\dots,z_n$ be distinct points in $Y'-Y$ and consider the subspace
\[
\widetilde{Y}=Y\times\{z_2\}\times\cdots\times\{z_n\}\subseteq C^n(Y')\subseteq X'.
\]
Again, the inequality in \eqref{eqn:SubspaceUB} shows $\TC_{X'}(\widetilde{Y}\times\widetilde{Y})\le \TC_{X'}(C^n(Y')\times C^n(Y')).$

Suppose $\TC_{X'}(\widetilde{Y}\times\widetilde{Y})=k,$ so that we may find open sets $\widetilde{U}_1,\dots,\widetilde{U}_k$ which cover $\widetilde{Y}\times\widetilde{Y}$ and admit sections $\tilde{s}_i\from \widetilde{U}_i\to P_{X'}(\widetilde{Y}\times\widetilde{Y}).$ Let $U_i= (p_1\times p_1)(\widetilde{U}_i),$ where $p_1\from (Y')^n\to Y'$ is the projection of the first factor. The sets $U_1,\dots, U_k$ form an open cover of $Y\times Y.$

We wish to define a section $s_i\from U_i\to P(Y).$  For each $y\in Y,$ let $\vec{y}=(y,z_2,\dots,z_n)\in \widetilde{Y}.$ Again, we get a continuous map $U_i\to \widetilde{U}_i$ sending $(x,y)$ to $(\vec{x},\vec{y}).$ Let $G\from X'\to Y$ be given by 
\[
G((y_1,t_1),\dots,(y_n,t_n))=r(y_1),
\]
and define $s_i\from U_i\to P(Y)$ by
\[
s_i(x,y)(t)=G(\widetilde{s}_i(\vec{x},\vec{y})(t)).
\]
Again, this is continuous for each $i,$ showing
\[
\TC(Y)\le k=\TC_{X'}(\widetilde{Y}\times\widetilde{Y})\le \TC_{X'}(C^n(Y')\times C^n(Y'))=\TC_X(C^n(Y)\times C^n(Y)).
\]
\end{proof}

\bibliographystyle{plain}
\bibliography{references}
\end{document}